\documentclass[12pt]{article}
\usepackage{amsmath,amssymb}
\usepackage[T1]{fontenc}
\usepackage{latexsym}
\usepackage{pdfsync}
\usepackage{epsfig}
\usepackage{proof}
\usepackage[utf8]{inputenc}
\usepackage[a4paper, top=3cm, bottom=3cm ]{geometry}
\newtheorem{thm}{Theorem}[section]

\newtheorem{lem}{Lemma}[section]

\newtheorem{prop}{Proposition}[section]

\title
{$L^p$ boundedness of the Bergman projection \\ on the generalized Hartogs triangles.}

\author{\normalsize Tomasz Beberok \\
\small Faculty of Mathematics and Computer Science, Jagiellonian University,\\
\small Lojasiewicza 6, 30-048 Krakow, Poland \\}

\date{}

\begin{document}

\begin{center}
  \textbf{$L^p$ boundedness of the Bergman projection \\ on the generalized Hartogs triangles
}
\end{center}
\vskip1em
\begin{center}
  Tomasz Beberok
\end{center}

\vskip2em

\noindent \textbf{Abstract.} In this paper we investigate a two classes of domains in $\mathbb{C}^n$ generalizing the Hartogs triangle. We prove optimal estimates for the mapping properties of the Bergman projection on  these domains.
\vskip1em

\textbf{Keyword:} Hartogs triangle; Bergman projection; Bergman kernel
\vskip1em
\textbf{AMS Subject Classifications:} 32W05; 32A25

\section{Introduction}
Let $\Omega \subset \mathbb{C}^n$ be a bounded domain. The Bergman space $L^2_a(\Omega)$ is defined to
be the intersection $L^2(\Omega) \cap  \mathcal{O}(\Omega)$ of the space $L^2(\Omega)$ of square integrable functions on $\Omega$ (with respect to the Lebesgue measure of $\mathbb{C}^n$) with the space $\mathcal{O}(\Omega)$ of holomorphic functions on $\Omega$. By the Bergman inequality, $L^2_a(\Omega)$ is a closed subspace of $L^2(\Omega)$. The orthogonal projection operator $\mathbf{P} \colon L^2(\Omega) \rightarrow L^2_a(\Omega)$ is the Bergman projection associated with the domain $\Omega$. It follows from the Riesz representation theorem that the Bergman projection is an integral operator with the kernel $K_{\Omega} (z, w)$ on $\Omega \times \Omega$, i.e. $\mathbf{P} f(z) = \int_{\Omega} K_{\Omega} (z, w)f(w)\,dV(w)$ for all $f \in L^2(\Omega)$. The Bergman kernel depends on the choice of $\Omega$ and is also represented by
\begin{align*}
K_{\Omega}(z,w)= \sum_{j=0}^{\infty} \phi_j(z) \overline{\phi_j(w)}, \quad (z,w) \in \Omega \times \Omega,
\end{align*}
where $\{\phi_j(\cdot) \colon  j = 0, 1, 2, . . .\}$ is a complete orthonormal basis for $L^2_a(\Omega)$. If $\Omega$ is the Hermitian unit ball $\mathbb{B}_n$ defined by
\begin{align*}
\mathbb{B}_n=\{z \in \mathbb{C}^n \colon |z_1|^2  + \ldots + |z_n|^2 < 1 \},
\end{align*}
it is easy to see that $z^{\alpha}$, $\alpha \in \mathbb{Z}_{+}^n $ form an orthogonal basis of $L^2_a(\mathbb{B}_n)$. A direct computation shows that $\| z^{\alpha} \| = \sqrt{ \frac{\alpha ! \pi^n }{(n + | \alpha|)!} }$. So the functions $\varphi_{\alpha} = \sqrt{ \frac{(n + | \alpha|)!}{\alpha ! \pi^n} } z^{\alpha}$, $\alpha \in \mathbb{Z}_{+}^n $, form an orthonormal basis of $L^2_a(\mathbb{B}_n)$. An easy computation gives:
    \begin{align}\label{ball}
        K_{\mathbb{B}_n}(z,w)=\frac{n!}{\pi^n} \frac{1}{(1- \langle z,w \rangle)^{n+1}},
    \end{align}
where $\langle z,w \rangle := z_1\overline{w}_1 + \ldots + z_n\overline{w}_n$.
Similarly as before, one can show that
\begin{align}\label{poydisc}
  K_{\mathbb{D}^n}(z,w)=\prod_{j=1}^{n}K_{\mathbb{D}}(z_j,w_j)= \frac{1}{\pi^n} \prod_{j=1}^{n} \frac{1}{(1-z_j \overline{w}_j)^{2}},
\end{align}
where $\mathbb{D}=\{z \in \mathbb{C} \colon |z|<1 \}$.
One of the features that makes the Bergman kernel both important and useful is its invariance under biholomorphic mappings. This fact is useful in conformal mapping theory, and it also gives rise to the Bergman metric. The fundamental result is as follows.
    \begin{prop}\label{biholo}
      Let $\Omega_1$ and $\Omega_2$ be domains in $\mathbb{C}^n$, and let $f \colon  \Omega_1 \rightarrow  \Omega_2$ be biholomorphic. Then
           \begin{align*}
             \det J_{\mathbb{C}} f(z) K_{\Omega_2}(f(z),f(w)) \det \overline{ J_{\mathbb{C}} f(w) }=K_{\Omega_1}(z,w).
           \end{align*}
    \end{prop}
\noindent Here $\det J_{\mathbb{C}} f$ is the complex Jacobian matrix of the mapping $f$ (see \cite{Kra} for more on this topic).
 \newline \indent
It is natural to consider the mapping properties of $\mathbf{P}$ on other spaces of functions on $\Omega$, for example the $L^p$ spaces. The $L^p$ mapping properties of the Bergman projection have been determined for large classes of domains we refer to the following articles and the references therein \cite{Ch,KP,LS,LS2,MN1,MN2,NRSW,PS,Zey}. In \cite{Ch}, Chakrabarti and Zeytuncu considered the classical Hartogs triangle and proved that the Bergman projection is a bounded operator from $L^p$ to $L^p_a$ if and only if $4/3<p<4$. In recent paper \cite{Edholm},  L.D. Edholm and J.D. McNeal generalized result obtained in \cite{Ch} to the domains defined by
\begin{align*}
  \{(z_1,z_2) \in \mathbb{C}^2 \colon |z_1|^k < |z_2|<1\},
\end{align*}
for every $k \in \mathbb{N}:=\{1,2,3,...\}$. Motivated by their work, in this paper we generalize their result to the following domains
\begin{align*}
  &\Omega_{k}:= \left\{(z,w) \in \mathbb{C}^{1+n} \colon   |w_1| < |z|^k <1,  \ldots,|w_1| < |z|^k <1  \right\}\\
  &\mathcal{H}_k:=\left\{(z,w) \in \mathbb{C}^{1+n} \colon   |w_1|^2+  \cdots + |w_n|^2 < |z|^{2k} <1  \right\},
\end{align*}
where again $k \in \mathbb{N}$. See \cite{Zap} for  the group of holomorphic automorphisms of $\mathcal{H}_k$ and proper holomorphic mappings between generalized Hartogs triangles.

\subsection{The Bergman kernel for  $\Omega_{k}$ and $\mathcal{H}_k$ }
Now we discuss the Bergman kernel for $\Omega_{k}$ and $\mathcal{H}_k$.
   \begin{thm} The Bergman kernel for $\Omega_{k}$ is given by
     \begin{align*}
       K_{\Omega_{k}}((z,w),(x,y))=  \frac{\eta^{nk}}{ \pi^{n+1} (1-\eta)^2  \prod_{j=1}^{n} (\eta^{k} - \nu_j)^2 },
     \end{align*}
   where $\nu_j=w_j \overline{y}_j$ for every $j=1,\ldots,n$ and $\eta=z\overline{x}$.
   \end{thm}
                     \begin{proof}
                     Using the biholomorphism $\phi \colon  \Omega_{k} \rightarrow \mathbb{D}^{*} \times \mathbb{D}^n$ defined by $\phi(z,w_1,\ldots,w_n):=(z,\frac{w_1}{z^k},\ldots,\frac{w_n}{z^k})$, well-known formula (\ref{poydisc}) and Proposition \ref{biholo} we can easily prove the desired result.
                     \end{proof}
Using the same function $\phi \colon  \mathcal{H}_{k} \rightarrow \mathbb{D}^{*} \times \mathbb{B}_n$, formula (\ref{ball}) and the behavior of the Bergman kernel under biholomorphic mappings we have
\begin{thm} The Bergman kernel for $\mathcal{H}_{k}$ is given by
     \begin{align}\label{kernelH}
       K_{\mathcal{H}_{k}}((z,w),(x,y))=  \frac{\eta^{k}}{ \pi^{n+1} (1-\eta)^2   (\eta^{k} - \nu_1- \cdots -\nu_n )^{n+1} }.
     \end{align}
   \end{thm}

\section{$L^p$ boundedness of the Bergman projection}
Before we give the main result of this work we prove the following lemma that will be used to prove the main theorem.
    \begin{lem}For every $\epsilon \in \left[\frac{1}{2}, \frac{kn+2}{2kn}\right)$, we have
           \begin{align}
             \int_{\Omega_{k}} \left|K_{\Omega_{k}}((z,w),(x,y))\right| h(x,y)^{-\epsilon} \, dV(x,y) \leq C h(z,w)^{-\epsilon} \label{est1}\\
            \int_{\mathcal{H}_k} \left|K_{\Omega_{k}}((z,w),(x,y))\right| g(x,y)^{-\epsilon} \, dV(x,y) \leq C_1 g(z,w)^{-\epsilon} \label{est2}
           \end{align}
    for some constants $C$ and $C_1$, where
            \begin{align*}
              h(x,y)&=(1-|x|^2)|x|^{2k(n-1)}(|x|^{2k}-|y_1|^2-\cdots - |y_n|^2) \\
              g(x,y)&=(1-|x|^2)(|x|^{2k} - |y_1|^2)\cdots(|x|^{2k} - |y_n|^2).
            \end{align*}
    \end{lem}
                      \begin{proof}
                        We start with inequality (\ref{est1}), by Theorem \ref{kernelH}
                        \begin{align*}
                           \int_{\Omega_{k}}& \left|K_{\Omega_{k}}((z,w),(x,y))\right| h(x,y)^{-\epsilon}  \, dV(x,y)\\
                           &=\int_{\Omega_{k}} \frac{|x|^{-2kn\epsilon} (1-|x|^2)^{-\epsilon} \left( 1 - \left|\frac{y_1}{x^k}\right|^2 - \cdots - \left|\frac{y_n}{x^k}\right|^2\right)^{-\epsilon} }{\pi^{n+1} |z \overline{x}|^{kn} |1-z \overline{x}|^2 \left| 1 - \frac{ w_1 \overline{y}_1 }{(z \overline{x})^k} - \cdots - \frac{  w_n \overline{y}_n }{(z \overline{x})^k  } \right|^{n+1}}  \, dV(x,y)
                        \end{align*}
                      Make the substitution $y_j = x^k u_j$ for every $j=1,\ldots,n$. This transformation sends $\Omega_{k}$ to $\mathbb{D}^{*} \times \mathbb{B}_n$
                       \begin{align*}
                         \int_{\mathbb{D}^{*} \times \mathbb{B}_n} \frac{|x|^{kn-2kn\epsilon} (1-|x|^2)^{-\epsilon} \left( 1 - \left|u_1\right|^2 - \cdots - \left|u_n\right|^2\right)^{-\epsilon} }{\pi^{n+1} |z|^{kn} |1-z \overline{x}|^2 \left| 1 - \frac{ w_1 }{z^k} \overline{u}_1 - \cdots - \frac{  w_n  }{z^k } \overline{u}_n   \right|^{n+1}}  \, dV(x,u)
                       \end{align*}
                      Next by Proposition 1.4.10 of \cite{R}
                              \begin{align*}
                           \int_{\Omega_{k}}& \left|K_{\Omega_{k}}((z,w),(x,y))\right| h(x,y)^{-\epsilon}  \, dV(x,y)\\
                           &\leq C \int_{\mathbb{D}}  \frac{|x|^{kn-2kn\epsilon} (1-|x|^2)^{-\epsilon} }{\pi^{n+1} |z|^{kn} |1-z \overline{x}|^2} \left( 1 - \left|\frac{ w_1 }{z^k}\right|^2  - \cdots - \left|\frac{  w_n  }{z^k }\right|^2   \right)^{-\epsilon}   \, dV(x)
                        \end{align*}
                      If $ kn -2kn\epsilon> -2$, then by Lemma (3.2) in \cite{Edholm}
                        \begin{align*}
                           \int_{\mathbb{D}}&  \frac{|x|^{kn-2kn\epsilon} (1-|x|^2)^{-\epsilon} }{\pi^{n+1} |z|^{kn} |1-z \overline{x}|^2} \left( 1 - \left|\frac{ w_1 }{z^k}\right|^2  - \cdots - \left|\frac{  w_n  }{z^k }\right|^2   \right)^{-\epsilon}   \, dV(x) \\ & \leq C' \frac{(1-|z|^2)^{-\epsilon}}{|z|^{kn}} \left( 1 - \left|\frac{ w_1 }{z^k}\right|^2  - \cdots - \left|\frac{  w_n  }{z^k }\right|^2   \right)^{-\epsilon}
                        \end{align*}
                      Finally when $2kn\epsilon \geq kn$ we obtain the desired result. The estimation (\ref{est2}) can be obtained by the similar method and we omit the details.
                      \end{proof}

  Now we are ready to formulate the  main result
       \begin{thm}\label{main}
         The Bergman projection is a bounded operator on  $L^p(\Omega_k)$  and $L^p(\mathcal{H}_{k})$ if and only if $p \in \left(\frac{2nk+2}{nk+2},\frac{2nk+2}{nk} \right)$.
       \end{thm}
\begin{proof}
 If $D$ is a Reinhardt domain, $f \in L^2_a(D) := \mathcal{O}(D) \cap L^2(D)$,  $f(z)=\sum_{\alpha \in \mathbb{Z}^n} a_{\alpha} z^{\alpha}$, then $\{  z^{\alpha} \colon \alpha \in \sum(f)  \} \subset L^2_a(D), $ where $\sum(f):=\{ \alpha \in \mathbb{Z}^n \colon a_{\alpha} \neq 0 \}$ (for proof see \cite{JP} p. 67). Thus it is easy to check, that the set $\{z^{\alpha} w_1^{\beta_1} \cdots w_n^{\beta_n} \colon   \beta_j \geq 0,  \alpha \geq   -k(|\beta|+n) \}$ is a complete orthogonal set for $L^2_a(\Omega_k)$ and $L^2_a(\mathcal{H}_k)$. Now it can be shown that
          \begin{align*}
            \mathbf{P}_{\Omega_k}(\overline{z}^{kn})(x,y)=\frac{C}{x^{nk}} \quad \text{and} \quad \mathbf{P}_{\mathcal{H}_k}(\overline{z}^{kn})(x,y)=\frac{C_1}{x^{nk}},
          \end{align*}
for some constants $C$ and $C_1$. Thus,
           \begin{align*}
             \|\mathbf{P}_{\Omega_k}(\overline{z})\|_p^p &\approx \int_{\Omega_k} \frac{1}{|x|^{pnk}} \,dV(x,y)  \\
              &=(2 \pi)^{n+1} \int_{0}^{1} \int_0^{r^{k}} \cdots \int_0^{r^{k}} r^{1-pnk} s_1 \cdots s_n \, ds_1 \cdots ds_n dr \\
              &=2\pi^{n+1} \int_{0}^{1} r^{1-pnk+2nk} \, dr
           \end{align*}
and
            \begin{align*}
             \|\mathbf{P}_{\mathcal{H}_k}(\overline{z})\|_p^p &\approx \int_{\mathcal{H}_k} \frac{1}{|x|^{pnk}} \,dV(x,y)  \\
              &=(2 \pi)^{n+1} \int_{0}^{1} \int_0^{r^{k}} \int_{S^{n-1}_{+}} r^{1-pnk} \rho^{2n-1} \omega  d\sigma(\omega) d\rho dr \\
              &=\frac{2\pi^{n+1}}{n!} \int_{0}^{1} r^{1-pnk+2nk} \, dr
           \end{align*}
These integrals diverges when $p\geq \frac{2+2nk}{nk}$, so $\mathbf{P}_{\Omega_k}(\overline{z}) \notin  L^p(\Omega_k)$ and $\mathbf{P}_{\mathcal{H}_k}(\overline{z}) \notin  L^p(\mathcal{H}_k)$ for this range of $p$. The fact that the Bergman projection is selfadjoint, together with H\"{o}lder's inequality, show that the $\mathbf{P}_{\mathcal{H}_k}$ and  $\mathbf{P}_{\Omega_k}$ also fails to be a bounded operator on $L^p(\Omega_k)$ and $L^p(\mathcal{H}_k)$, respectively when $p \in (1, \frac{2kn+2}{kn+2}]$. In order to prove boundedness it is enough to combine Lemma 3.1 and Schur's lemma (Lemma 2.4 in \cite{Edholm}) these yield that the operators $L^p(\Omega_k)$ and $L^p(\mathcal{H}_k)$ are bounded for $p \in \left(\frac{2nk+2}{nk+2},\frac{2nk+2}{nk} \right)$.
\end{proof}
\subsection{Concluding remarks}
    \begin{enumerate}
      \item Putting together Theorem \ref{main} and Theorem 3.1 from \cite{Edholm}, we have three different domains  $\Omega_{k}$, $\mathcal{H}_k$ and
        \begin{align*}
         \mathbb{H}_{nk}:=\left\{(z,w) \in \mathbb{C}^2  \colon  |w|^{nk} < |z| \right\}
        \end{align*}
        for which Bergman projection is bounded with the same range of $p$.
      \item For $\Omega_{k}$ and $\mathcal{H}_k$ we can also prove similar results to those obtained in \cite{Ch} related to weighted $L^p$ mapping behavior.
      \item It would be interesting to investigate a class of domains considered by Zapa{\l}owski \cite{Zap} which generalize the classical Hartogs triangle. Domains of the form
            \begin{align*}
              \left\{(z,w) \in \mathbb{C}^n \times \mathbb{C}^m \colon \sum_{j=1}^{n} |z_j|^{2p_j} < \sum_{j=1}^{m} |w_j|^{2q_j} <1 \right\}.
            \end{align*}
    \end{enumerate}

\bibliographystyle{amsplain}

\noindent Tomasz Beberok\\
Sabanci University\\
Orhanli, Tuzla 34956\\
Istanbul, Turkey\\
email: tbeberok@ar.krakow.pl

% ------------------------------------------------------------------------
\end{document}